\documentclass[12pt]{amsart}

\usepackage{graphicx}

\usepackage{tikz-cd}

\usepackage{pgf}

 \usepackage{url}	
 \allowdisplaybreaks 

\usepackage{tikz-cd}
\usepackage{pgf}

\newtheorem{teo}{Theorem}[section]
\newtheorem{theorem}[teo]{Theorem}
\newtheorem{cor}[teo]{Corollary}
\newtheorem{corollary}[teo]{Corollary}

\newtheorem{lema}[teo]{Lemma}
\newtheorem{lemma}[teo]{Lemma}
\newtheorem{prop}[teo]{Proposition}
\newtheorem{proposition}[teo]{Proposition}

\theoremstyle{definition}
\newtheorem{definition}[teo]{Definition}
\newtheorem{defi}[teo]{Definition}
\newtheorem{example}[teo]{Example}
\newtheorem{ex}[teo]{Example}

\theoremstyle{remark}
\newtheorem{remark}{Remark}

\numberwithin{figure}{section}

\newcommand{\C}{\mathcal{C}}
\newcommand{\I}{\mathcal{I}}
\newcommand{\Dc}{\mathcal{D}}
\newcommand{\D}{\mathrm{D}}
\newcommand{\E}{\mathcal{E}}
\newcommand{\U}{\mathcal{U}}

\newcommand{\cat}{\mathop{\mathrm{cat}}}
\newcommand{\ccat}{\mathop{\mathrm{ccat}}}
\newcommand{\wccat}{\mathop{\mathrm{wccat}}}

\newcommand{\cTC}{\mathop{\mathrm{cTC}}}
\newcommand{\sTC}{\mathop{\mathrm{sTC}}}
\newcommand{\sD}{\mathop{\mathrm{sD}}}
\newcommand{\scat}{\mathop{\mathrm{scat}}}

\newcommand{\cD}{\mathrm{cD}}
\newcommand{\wcD}{\mathrm{wcD}}

\newcommand{\sd}{\mathrm{sd}}
\newcommand{\ob}{\mathrm{Ob}}
\newcommand{\arr}{\mathrm{Arr}}
\newcommand{\B}{\mathrm{B}}
\newcommand{\op}{\mathrm{op}}

\newcommand{\co}{\colon}

\begin{document}
\title{Homotopic distance between functors}
\thanks{The first author was partially supported by MINECO-FEDER research project MTM2016-78647-P. The second author was partly supported by Ministerio de Ciencia, Innovaci\'on y Universidades,  grant FPU17/03443.}
\author[]{E. Mac\'{i}as-Virg\'os}
\address[E. Mac\'{i}as-Virg\'os, D. Mosquera-Lois] {\newline \indent Institute of Mathematics, University of Santiago de Compostela, Spain.}
\email{quique.macias@usc.es} 
\email{david.mosquera.lois@usc.es}

\author[]{D. Mosquera-Lois}

\begin{abstract} We introduce a notion of {\em categorical homotopic distance between functors} by adapting the notion of homotopic distance in topological spaces, recently defined by the authors to the context of small categories. Moreover, this notion generalizes the work on categorical LS-category of small categories by Tanaka. 
\end{abstract}


\maketitle

\section{Introduction}\label{INTRO}

Recently, some topological concepts were extended to small categories, this is the case of the Euler characteristic by Leinster \cite{Leinster_2,Leinster} and both a notion of Lusternik-Schnirelman category \cite{Tanaka3} and a theory of Euler Calculus in the context of small categories  by Tanaka \cite{Tanaka_integration,Tanaka_networks}. Moreover, the authors have generalized both the LS-category and the Topological Complexity by means of a new notion of homotopic distance between continuous maps \cite{QuiDa}. The purpose of this work is to adapt the  notion of homotopic distance  to the context of functors between small categories. Furthermore, this ``homotopic distance between functors'' generalizes the categorical LS-category introduced by Tanaka \cite{Tanaka3} and allows us to define a notion of ``topological complexity for categories'', which may be thought as an adaptation of the Topological complexity introduced by Farber \cite{Farber} to the context of categories. 

The organization of the paper is as follows:

 In Section \ref{sec:cat_dis_between_funct} we recall the well known definitions of homotopy and weak homotopy for functors between categories   and then we introduce two corresponding definitions of categorical  distance, which we call  ``categorical homotopic distance'' and ``weak categorical homotopic distance'' between functors, respectively. 
  
 Section \ref{sec:examples_and_computations} is devoted to present particular cases of categorical homotopic distance such as the categorical LS-category introduced by Tanaka \cite{Tanaka3} and a new notion of ``topological complexity'' for categories. 
 
 In Section \ref{sec:properties} we prove several properties of the categorical homotopic distance such as its behavior under compositions and products and its homotopical invariance. Moreover, we prove that the homotopic distance between functors is bounded above by the category of the domain. Afterwards, we relate the two  notions of homotopic distance between the functors $F,G$ to the homotopic distance of the  continuous maps
 $\B F,\B G$ associated by the classifying space functor. 
 
 In Section \ref{sec:Posets} we restrict our attention to the setting of posets, when seen as small categories. This setting is important for several reasons: first, stronger statements can be made; second, their appearance in several applications \cite{Tanaka_networks}; third, the fact that posets and small categories are strongly related by subdivision functors \cite{Hoyo}.

\subsection*{Acknowledgements} The second author thanks Victor Carmona S\'anchez for enlightening conversations and discussions.

\section{Categorical distance between functors}\label{sec:cat_dis_between_funct} 
	
In this section we recall the notions of homotopy and weak homotopy between functors. Then we introduce the (weak) categorical homotopic distance between functors. For more details on some classical definitions and constructions we refer the reader  to \cite{LeeHomot,LeeHomot2,May_simplicial,Minian,Quillen,Segal}.

\subsection{Homotopies between functors}
	
	Given a small category $\mathcal{C}$, we denote by $\ob(\C)$ its set of objects, by $\arr(\C)$ its set of arrows and by $\C(x,y)$ the set of arrows between the objects $x$ and $y$. 

	All categories will be assumed to be small and all functors will be assumed to be covariant unless stated otherwise. We begin by introducing the notion of homotopy between functors \cite{LeeHomot,LeeHomot2}.
	
	\begin{definition}\label{def:interval_category}
	The {\em interval category $\mathcal{I}_m$} of length $m\geq 0$ consists of $m+1$ objects with zigzag arrrows formed of \[0 \rightarrow 1 \leftarrow 2 \rightarrow \cdots \rightarrow (\leftarrow) m.\]
	Alternatively, the interval category $\mathcal{I}_m$ can be defined in the following way: the objects of $\mathcal{I}_m$ are the non-negative integers $0, 1, \dots, m$ and the arrows, other than the identities, are defined as follows. Given two distinct objects $r$ and $s$ in $\mathcal{I}_m$, there is exactly one arrow from $r$ to $s$ if $r$ is even and $s = r-1$ or $s = r + 1$, and no arrows otherwise.
	\end{definition}
	
	Given two small categories $\C$ and $\Dc$ we denote its product by $\C \times \Dc$. Recall that the objects of  $\C \times \Dc$ are pairs of objects in $\C$ and objects in $\Dc$, and its arrows are products of arrows in  $\C$ and arrows in $\Dc$.
	
	\begin{definition}
		Let $F,G\colon \C\to \Dc$ be two functors between small categories. We say that $F$  and $G$ are {\em homotopic} if there exists a functor $H\co \C \times \mathcal{I}_m \rightarrow \D$, called a homotopy (with length $m$), such that $H_0=F$ and $H_m=G$, for some $m\geq 0$. 
	\end{definition}
	
	Alternatively, the notion of homotopy between functors can be defined as follows. Both definitions are equivalent.  
	
	\begin{definition}
		Let $F,G\colon \C\to \Dc$ be two functors between small categories. We say that $F$  and $G$ are {\em homotopic}, $F \simeq G$, if there is a finite sequence of functors $F_0,\dots,F_m\colon \C\to \Dc$, with $F_0=F$ and $F_m=G$, such that for each $i\in \{0,\ldots,m-1\}$ there is a natural transformation between $F_i$ and $F_{i+1}$ or between $F_{i+1}$ and $F_i$. 
	\end{definition}
	
		Homotopies can be concatenated \cite{Tanaka3} and therefore, the homotopy relation between functors defined above is an equivalence relation. It also holds that the the relation behaves well with respect to compositions, i.e., if $F\simeq F'$ and $G \simeq G'$, then  $F \circ G \simeq F' \circ G'$ whenever $F \circ F'$ and $G \circ G'$ make sense.

	In order to state the next definition we briefly recall the definition of the classifying space functor $\B$  from small categories to topological spaces. 		 
	Given the small category $\C$, its {\em nerve $\mathrm{N}\C$} is a simplicial set whose $m$-simplices are 
	composable $m$-tuples of arrows in $\C$: $$c_0\xrightarrow{\alpha_1} \cdots \xrightarrow{\alpha_m} c_m.$$
	The face maps are obtained by composing or deleting arrows and the degenerate maps are obtained by inserting identities. A $m$-simplex of $\mathrm{N}\C$ is called {\em non-degenerate} if it includes no identity.  Given a functor $F\colon \C \to \Dc$ between small categories, we define $\mathrm{N} F\colon \mathrm{N}\C \to \mathrm{N}\Dc$ as follows: if $c_0\xrightarrow{\alpha_1} \cdots \xrightarrow{\alpha_m} c_m$ is a  $m$-simplex in $\mathrm{N}\C$, then 
	$$\mathrm{N} F(c_0\xrightarrow{\alpha_1} \cdots \xrightarrow{\alpha_m} c_m)=F(c_0)\xrightarrow{F(\alpha_1)} \cdots \xrightarrow{F(\alpha_m)} F(c_m).$$
	 The {\em classifying space $\B \C$} is then the geometric realization $\vert \mathrm{N}\C \vert$ of the  simplicial set $\mathrm{N}\C$. Moreover, $\B \C$  is a CW-complex with one $m$-cell for each non-degenerate $m$-simplex of $\mathrm{N}\C$. This construction is functorial \cite{May_simplicial, Quillen}, because  given a map  $\phi \colon K \to L$ between simplicial sets,  its geometric realization is a continuous map between topological spaces $\vert \phi \vert \colon \vert K\vert \to \vert L\vert$.  The classifying space functor is defined as the composition of the nerve functor with the geometric realization functor.

	\begin{definition}
	Let $F,G\colon \C\to \Dc$ be two functors between small categories. We say that $F$  and $G$ are {\em weak homotopic}, denoted $F \simeq_{w} G$ if the maps $\B F,\B G \colon\B \C\to \B \Dc$ are homotopic.
	\end{definition}

	\begin{definition}
	A functor $F\colon \C\to \Dc$ is said to be a  homotopy equivalence (respectively a weak homotopy equivalence) if there exists another functor $G\co \Dc \to \C$ such that  $G\circ F\simeq 1_{\C}$ (respectively $G\circ F\simeq_w 1_{\C}$) and  $F\circ G\simeq 1_{\Dc}$ (respectively $F\circ G\simeq_w 1_{\Dc}$). Under these circumstances we say that the categories $\C$ and $\Dc$ are  homotopy equivalent (respectively weak homotopy equivalent).
	\end{definition}

	Recall that the classifying space functor preserves homotopies, that is, if two functors $F,G\colon \C\to \Dc$ are homotopic, then the induced maps $\B F,\B G\colon \B\C\to \B\Dc$ on the classifying spaces are also homotopic. Therefore, homotopy equivalence between categories implies weak homotopy equivalence. However, the converse does not hold as the following example given by Minian \cite{Minian} shows:
	
	\begin{example}\label{ex:poset_revirado_numeros_naturales}
	Consider a category $\mathcal{N}$ whose objects are the non-negative integers and the arrows, other than the identities, are defined as follows. If $r$ and $s$ are two distinct objects in $\mathcal{N}$, there is exactly one arrow from $r$ to $s$ if $r$ is even and $s = r-1$ or $s = r + 1$ and no arrows otherwise. Assume there is a functor $F\colon \mathcal{N} \to \mathcal{N}$ such that $F\simeq 1_{\mathcal{N}}$. We claim that there exists a non-negative integer $n_0$ such that $F(n)=n$, for all $n\geq n_0$, in particular $F$ is not constant and the category $\mathcal{N}$ is not contractible. 
	
	Let us prove the claim.  First, note that if there exists a natural transformation $G\Rightarrow 1_{\mathcal{N}}$ or $1_{\mathcal{N}}\Rightarrow G$, then $G$ fixes the odd numbers, that is, $G(n)=n$ for $n$ odd. As a consequence, it follows that $G(m)=m$ for every $m>0$. By repeating a similar argument it can be deduced that if there exists a natural transformation $G'\Rightarrow G$ or $G\Rightarrow G'$, then $G'$ fixes all natural numbers larger than one. Repeating this argument it follows that if $F\simeq 1_{\mathcal{N}}$, then  there exists a non negative integer $n_0$ such that $F(n)=n$ for all $n\geq n_0$ as we claimed. 
	
	However, the category $\mathcal{N}$ is weak contractible since $\B \mathcal{N}$ is homotopy equivalent to $[0,+\infty)$ and the map $\B \mathcal{N}\to \B (0)$ is a homotopy equivalence of topological spaces.
	\end{example}

	\begin{example}
		When the small categories $\C$ and $\Dc$ are partially ordered sets seen as finite topological spaces \cite{Barmak_book} (see Section \ref{sec:Posets}) and $F,G\colon \C\to \Dc$ are order preserving maps, then the notion of homotopy between functors is equivalent to the usual notion of homotopy in the context of topological spaces \cite{Raptis}. 
	\end{example}

We recall a notion of ``connectedness'' for categories \cite{Riehl_2}:

\begin{definition}\label{def:connectedness_for_categories}
A small category $\C$ is said to be {\em connected} if for any pair of objects $c,c'$ there is a finite sequence of zigzag arrows joining them: \[c=c_0 \rightarrow c_1 \leftarrow c_2 \rightarrow \cdots \rightarrow (\leftarrow) c_m=c'.\]
\end{definition}

Equivalently, a small category $\C$ is said to be connected if for any pair of objects $c,c'$ there is a functor $F$ from some interval category $\mathcal{I}_m$ to $\C$ such that $F(0)=c$ and $F(m)=c'$.

\begin{definition}
Given two small categories $\C$ and $\Dc$ it is said that a functor $d_0\colon \C\to \Dc$ is a {\em constant functor} onto the object $d_0$ of $\Dc$ if $d_0\colon \C\to \Dc$ takes every object of $\C$ to $d_0$ and every arrow to the identity arrow of $d_0$. 
\end{definition}

\begin{remark}
	Observe that a small category $\C$ is connected if and only if any pair of constant functors $c_0,c_1\colon \C\to \C$ onto objects $c_0$ and $c_1$ are homotopic.  
\end{remark}

From now on all categories are assumed to be connected.

\begin{example}
In the context of posets, Definition \ref{def:connectedness_for_categories} corresponds to the notion of order-connectedness (which is equivalent to topological connectedness for the associated finite spaces \cite{Barmak_book}). 
\end{example}

\begin{definition}
A small category $\C$ is said to be {\em contractible} if the identity functor is homotopic to a constant functor onto an object.
\end{definition}
	

We state a useful result.

\begin{proposition}\label{prop:adjoints_are_homotopy_equiv}
Suppose the functor $F\colon \C\to \Dc$ between small categories has a left or right adjoint $G\colon \Dc\to \C$. Then $F\colon \C\to \Dc$ is a homotopy equivalence. In particular, when $\C$ has an initial or terminal object, $\C$ is contractible.
\end{proposition}

\begin{proof}
	If $G$ is a right adjoint to $F$, then there are two natural transformations $1_{\C} \Rightarrow G\circ F$ and $F\circ G \Rightarrow 1_{\Dc}$. Therefore $F$ is a homotopy equivalence.
\end{proof}

\subsection{Definition of categorical distance}

We begin by introducing a suitable notion of covers for categories in order to define a notion of categorical distance between functors. The idea for this approach comes from thinking on the associated cover of the classifying space in order cover the arrows of the category. It was introduced by Tanaka~\cite{Tanaka3}.

\begin{definition}
	A collection of subcategories $\{\U_{\lambda}\}_{\lambda\in \Lambda}$ of a category $\C$ is a {\em geometric cover of $\C$} if for every sequence of composable arrows $f_1,\ldots,f_n$ in $\C$, there exists an index $\lambda \in \Lambda$ such that every $f_i$ belongs to $\U_{\lambda}$.
\end{definition}

Recall from \cite{Tanaka3}:

\begin{prop}\label{prop:condition_being_geometric_cover}
	Let $\{\U_{\lambda}\}_{\lambda \in \Lambda}$ be a collection of subcategories of a category $\C$. This is a geometric cover if and only if the collection of subcomplexes $\{\B \U_{\lambda}\}_{\lambda \in \Lambda}$ covers $\B\C$.
\end{prop}

Now we introduce our definition of ``distance'':

\begin{definition}Let $F,G\colon \mathcal{C} \to \mathcal{D}$ be two functors between small categories. The {\em categorical homotopic distance} $\cD(F,G)$ between $F$ and $G$  is the least integer $n\geq 0$ such that there exists a geometric cover $\{\U_0,\dots,\U_n\}$ of $\mathcal{C}$ with the property that  $F_{\vert \U_j}\simeq G_{\vert \U_j}$,   for all $j=0,\dots,n$.  If there is no such covering, we define $\cD(F,G)=\infty$.
\end{definition}

\begin{ex}
	Any finite group $\mathcal{G}$ can be seen as a category with only one object, where the arrows are the elements of $G$. Then it can be checked easily that if $\mathcal{G}$ is a non trivial group and $F,G\co \mathcal{G} \to \mathcal{G}$ are two functors, that is, two group homomorphisms, then $cD(F,G)=\infty$ unless $F=G$.
\end{ex}

It is easy to prove that some properties of the homotopic distance for continuous maps also hold for the categorical homotopical distance:

\begin{enumerate}
	\item \label{UNO}
	$\cD(F,G)=\cD(G,F)$.
	\item \label{DOS}
	$\cD(F,G)=0$ if and only if the functors $F,G$ are homotopic. 
	\item \label{TRES}
	 The categorical homotopic distance only depends on the homotopy class, that is,
	if $F\simeq F^\prime$ and $G\simeq G^\prime$ then $\cD(F,G)=\cD(F^\prime,G^\prime)$.
	\item \label{CUATRO}
	Given two functors $F,G\colon \C\to \Dc$ and a finite geometric covering $\U_0,\dots,\U_n$ of $\C$, it is $$\cD(F,G)\leq \sum_{k=0}^n\cD(F_{\vert \U_k},G_{\vert \U_k})+ n.$$ 
	\item A small category $\C$ is connected if and only if the categorical homotopic distance between any pair of constant functors is zero. 
\end{enumerate}

Now it comes our second definition of ``distance between functors'':

\begin{definition}The {\em weak categorical homotopic distance} $\wcD(F,G)$ between $F$ and $G$  is the least integer $n\geq 0$ such that there exists a geometric covering $\{\U_0,\dots,\U_n\}$ of $\mathcal{C}$ with the property that  $\B F_{\vert \B \U_j}\simeq \B G_{\vert \B \U_j}$,   for all $j=0,\dots,n$.  If there is no such covering, we define $\wcD(F,G)=\infty$.
\end{definition}

The weak homotopic distance satisfies the analogous statements to Properties \eqref{UNO}--\eqref{CUATRO} above of the categorical homotopical distance. Moreover, both the weak categorical and the categorical distance behave well with respect to duality: 

\begin{itemize}
	\item[(6)] \label{CINCO} Given a functor $F\colon \C\to \Dc$ we can define $F^{\op}\colon \C^{\op}\to \Dc^{\op}$. Moreover, if there is a natural transformation between $F$ and $G$, then there is a natural transformation between $G^{\op}$ and $F^{\op}$. Therefore, $\cD(F,G)=\cD(F^{\op},G^{\op})$. Notice that it holds $\B F=\B F^{\op}$. Hence, $\wcD(F,G)=\wcD(F^{\op},G^{\op})$. 
\end{itemize}

%
%

\section{Examples}\label{sec:examples_and_computations}

Recall that all categories are assumed to be small and connected.

\subsection{Categorical LS category}

We begin by restating the concept of categorical Lusternik-Schnirelmann introduced by Tanaka \cite{Tanaka3} as a particular case of the more general notion of categorical homotopic distance:

\begin{definition}
Let $\C$ be a small category. A subcategory $\U$ is {\em categorical in $\C$} if the inclusion functor is homotopic to a constant functor onto an object. The {\em(normalized) categorical Lusternik-Schnirelmann $\ccat(\C)$} is the least integer $n\geq 0$ such that there exists a geometric cover of $\C$ formed by $n+1$ categorical subcategories. If there is no such an integer we set $\ccat(\C)=\infty$.
\end{definition}

The following result is just a reformulation of the definition  of the categorical Lusternik-Schnirelmann:

\begin{prop}
The LS-category of $\C$ is the categorical homotopic distance between the identity $1_\C$ of $\C$ and any constant functor, that is $\cat(X)=\D(1_\C,*)$.
\end{prop}

More generally, we define the categorical Lusternik-Schnirelmann category of a functor: 

\begin{defi}
The {\em (weak) categorical Lusternik-Schnirelmann category of the functor $F \colon \C\to \Dc$} is the (weak) categorical distance bewteen $F$ and a constant functor, $\cD(F)=\cD(F,*)$. 
\end{defi}

\begin{ex}
The category of the diagonal functor $\Delta_{\C}\colon \C \to \C \times \C$ equals $\ccat(X)$. 
\end{ex}

Given a base object $c_0\in \C$ we define the inclusion functors $i_1,i_2\colon \C \to \C \times \C$ as $i_1(c)=(c,c_0)$ and $i_2(c)=(c_0,c)$. 

\begin{prop}\label{INCLCAT}
	The categorical LS-category of $\C$ equals the categorical homotopic distance between $i_1$ and $i_2$, that is, $\ccat(\C)=\cD(i_1,i_2)$.
\end{prop}

\begin{proof}
	First, we show that $\cD(i_1,i_2)\leq \ccat(X)$. Assume that a subcategory $\U$ of $\C$ is categorical and let  $H\colon \U \times \I_m \to \C$ be the homotopy between the inclusion functor and the constant functor to $c_0\in \C$, i.e. $H(c,0)=c$ and $H(c,1)=c_0$. We define a homotopy $H^\prime\colon \U \times \I_{2m} \to \C$ between $(i_1)_{\vert \U}$ and $(i_2)_{\vert \U}$ (by concatenation) as
	$$H^\prime(c,i)=\begin{cases}
	\big(H(c,i),c_0\big) & \text{if\ } 0\leq i\leq m,\\
	\big(c_0,H(c,2m-i)\big)  &\text{if\ }m\leq i \leq 2m.\end{cases}$$
	Note that:
	$$H^\prime(c,0)=\big(H(c,0),c_0\big)=(c,c_0)=i_1(c)$$
	while
	$$H^\prime(c,2m)=\big(c_0,H(c,0)\big)=(c_0,c)=i_2(c).$$
	Second, we show that $\ccat(X)\leq \cD(i_1,i_2)$. Assume that there is a homotopy $H\colon \U\times \I_m\to \C \times \C$ between $(i_1)_{\vert \U}$ and $(i_2)_{\vert \U}$, i.e., $H(c,0)=(c,c_0)$ and $H(x,1)=(c_0,c)$. Let  $p_1\circ H$ be the first component of $H$. Then $p_1\circ H$  is a homotopy between the inclusion functor of $\U$ and the constant functor onto $c_0$.  
\end{proof}

\subsection{Categorical complexity of a category}

Motivated by the approach adopted by one of the authors to the Discrete Topological Complexity in the setting of simplicial complexes in \cite{SamQuiJa_TC}, we define the complexity of a category as follows:

\begin{definition}
A subcategory $\U$ of $\C \times \C$ is a {\em Farber subcategory} if there exists a functor $F\co \U\to \C$ such that $\Delta \circ F \simeq i_{\U}$ where $i_{\U}$ is the inclusion functor. The {\em(normalized) categorical complexity of $\C$, $\cTC(\C)$}, is the least integer $n\geq 0$ such that there exists a geometric cover of $\C$ formed by $n+1$ Farber subcategories. If there is no such an integer we set $\cTC(\C)=\infty$.
\end{definition}

\begin{theorem}\label{PROJECT}
The categorical complexity of a small category $\C$ is the categorical homotopic distance between the two projections $p_1,p_2\colon \C \times \C \to \C$, that is,
$\cTC(\C)=\cD(p_1,p_2)$.
\end{theorem}

\begin{proof}
We will prove that a subcategory $\U$ of $\C \times \C$ is a Farber subcategory if and only if the projection functors are homotopic in  $\U$.  First, assume that there exists a functor $F\co \U\to \C$ such that $\Delta \circ F \simeq i_{\U}$.  Let us denote the homotopy between $\Delta \circ F$ and $i_{\U}$ by $H\colon \U \times \mathcal{I}_m\to \C \times \C$, where 
$$H_0(c_1,c_2)=(\Delta \circ F)(c_1,c_2)=(F(c_1,c_2),F(c_1,c_2))$$ and $H_1(c_1,c_2)=(c_1,c_2)$. We define a homotopy $H^\prime\colon \U\times {I}_{2m}\to \C$ between the projection functors as follows:
	$$H'(c_1,c_2,i)=\begin{cases}
	p_1\circ H(c_1,c_2,m-i) & \text{if\ } 0\leq i\leq m,\\
	p_2 \circ H(c_1,c_2,i-m)  &\text{if\ }m\leq i \leq 2m.
	\end{cases}$$
	Conversely, assume that the projection functors are homotopic in  $\U$ through a homotopy $H^\prime\colon \U \times \mathcal{I}_m\to \C$ where $H^\prime_0=p_1$ and $H^\prime_m=p_2$ and we will prove that there exists a functor $F\co \U\to \C$ such that $\Delta \circ F \simeq i_{\U}$. Define $F=p_1$. Now, the homotopy between  $\Delta \circ F$ and $i_{\U}$ is given by $G\colon \U \times \mathcal{I}_m\to \C \times \C$, where $G(c_1,c_2,m)=(c_1,H^\prime(c_1,c_2,m))$.
\end{proof}

\section{Properties}\label{sec:properties}

Recall that all categories are assumed to be small and connected.

\subsection{Compositions} We prove several elementary properties, starting with the behaviour of the homotopic distance under compositions. Several properties of $\ccat$ and $\cTC$ can be deduced from our general results.

\begin{proposition}\label{IZQ}
Let be functors $F,G\colon \C \to \Dc$ and $H\colon \Dc \to \E$. Then 
$$\cD(H\circ F,H\circ G)\leq \cD(F,G).$$
\end{proposition}

\begin{proof}Let $\cD(F,G)\leq n$ and let $\{\U_0,\dots, \U_n\}$ be a geometric covering of $\C$ with $F_j=F_{\vert U_j}$ homotopic to $G_j=G_{\vert \U_j}$. Then 
$$(H\circ F)_j=H\circ F_j\simeq H\circ G_j=(H\circ G)_j,$$ so $\cD(H\circ F,H \circ G)\leq n$.
\end{proof}

\begin{cor}\label{cor_d_upper_bound_domain}
Let $F\colon \C \to \Dc$ be a functor. Then $\ccat(F)\leq \ccat(\C)$. 
\end{cor}

\begin {proof}Take $1_{\C}$ and a constant functor $c_0$ from $\C$ to $\C$. Then
$\D(F\circ1_{\C},F(c_0))\leq \D(1_{\C},c_0)$.
\end{proof}

\begin{proposition}\label{DER}
Let be functors $F,G\colon \C \to \Dc$ and $H\colon \E\to \C$. Then 
$$\cD(F\circ H,G\circ H)\leq \cD(F,G).$$
\end{proposition}

\begin{proof}Let $\cD(F,G)\leq n$ and let $\{\U_0,\dots, \U_n\}$ be a geometric covering of $\C$ with $F_j\simeq G_j\colon \U_j \to \Dc$. Since $\C$ is a small category, for each $\U=\U_j$ we can define the subcategory $H^{-1}(\U)$ where 
	$$\ob(H^{-1}(\U))=\{e\in \ob(\E) \co H(e)\in \ob(\U)\}$$ and if $e,e^\prime\in\ob(H^{-1}(\U))$ then
	$$\arr(e,e^\prime)=\{\alpha\in \arr(\E) \co H(\alpha)\in \arr_{\U}(h(e),h(e^\prime))\}.$$ 
	Consider the geometric covering of $\E$ whose elements are $V_j=H^{-1}(\U_j)$. The restriction $H_j\colon V_j \to \C$ can be written as the composition of  $\bar H_j\colon V_j\to U_j$, where $\bar H_j(c)=H(c)$, and the inclusion $I_j$ of $U_j$ in $\C$. Then we have that 
$$(F\circ H)_j=F_j\circ \bar H_j \simeq G_j\circ \bar H_j=G\circ I_j\circ \bar H_j=G\circ H_j=(G\circ H)_j,$$
 hence $\cD(F\circ H,G\circ H)\leq n$.
\end{proof}

\begin{cor}
Given a functor $F\colon \C \to \Dc$, then $\ccat(F)\leq \ccat(\Dc)$.
\end{cor}

\begin{proof}Take $1_{\Dc}$ and a constant functor $d_0$ from $\Dc$ to $\Dc$.
Then
$\cD(1_{\Dc}\circ F, d_0\circ F)\leq \cD(1_{\Dc},y_0)$.
\end{proof}

The latter result can be extended. 

\begin{cor}\label{cor:d_upper_bound_codomain}
Let $F,G\colon \C \to \Dc$ be functors. Then 
$$\cD(F,G)+1\leq (\ccat (F)+1)(\ccat(G)+1).$$
\end{cor}

\begin{proof}
Denote by $d_0$ a constant functor from $\C$ to $\Dc$.  Assume that $\ccat(F)=\cD(F,d_0)\leq m$, $\ccat(G)=\D(G,d_0)\leq n$ and let $\{\U_0,\ldots,\U_m\}$, $\{\mathcal{V}_0,\ldots,\mathcal{V}_n\}$ be the corresponding geometric coverings of $\C$. The subcategories $W_{i,j}=\U_i\cap \mathcal{V}_j$ (where the intersection means the intersections of the sets of objects and intersections of the sets of arrows) form a geometric cover of $\C$.   Moreover, $F \simeq d_0\simeq G$ on $W_{i,j}$, so $cD(F,G)\leq m\cdot n$ - 1. The result follows.
\end{proof}

\begin{cor}\label{TCCAT1} 
$\ccat(\C)\leq \cTC(\C)$.
\end{cor}   

\begin{proof}In Proposition \ref{DER} consider the inclusion funtors $i_1,i_2\colon \C\to \C\times \C$, so $$\cD(*,1_{\C})=\cD(p_1\circ i_2,p_2\circ i_2)\leq \cD(p_1,p_2).\qedhere$$
\end{proof}

\subsection{Domain and codomain}

\begin{proposition}\label{prop_terminal_initial_object_implies_zero_homotopic_distance}
Assume that $F,G\colon \C\to \Dc$  are two functors between small categories. If at least one of the categories $\C$ or $\Dc$ have an initial or terminal object, then $\cD(F,G)=0$. 
\end{proposition}

\begin{proof}
Recall from Proposition \ref{prop:adjoints_are_homotopy_equiv} that a category $\E$ that has an initial or terminal object, is contractible. Since any pair of functors which contractible domain or codomain are homotopic, it is $\cD(F,G)=0$.
\end{proof}

\begin{remark}
The converse of Proposition \ref{prop_terminal_initial_object_implies_zero_homotopic_distance} does not hold. Recall that a poset $P$, when seen as a category (see Section \ref{sec:Posets}), has a terminal object $x$ if and only if $x$ is the unique maximal element of $P$ and a dual statement applies to initial elements. Consider the poset $P=\{x,y,z,w,t\}$ with the order $x\leq z,w,t$, $y\leq z,w,t$ and $z\leq w,t$. It is contractible and therefore the distance between any pair of functors is zero. However, it has no terminal nor initial objects.
\end{remark}

\begin{theorem}\label{CATDOM}
Let $F,G \colon \C \to \Dc$ be two functors. Then
$$ \cD(F, G)\leq \ccat(C).$$
\end{theorem}

%
%
%

\begin{proof} It is enough to prove that  $$\cD(F,G)=\cD(F\circ 1_{\C},G\circ 1_{\C})\leq \cD(1_{\C},c_0)=\ccat(X).$$
	Assume $\cD(1_{\C},c_0)= n$, and let $\{\U_0,\ldots,\U_n\}$ be a geometric covering for $\C$ such that, for all $j$,  $1_{\vert \U_j}\simeq (c_0)_{\vert \U_j}$ by a homotopy $H\colon \U_j\times \I_{m} \to \C$. Let us define the homotopy $H^\prime\colon \U_j\times I_{2m} \to \Dc$ as follows:
		\begin{equation*}\label{FORMULA}
	H^\prime(c,t)= 
	\begin{cases}    
	F \circ \mathcal{H}(c,i), & \text{if\ } 0 \leq i \leq m, \\
	G \circ \mathcal{H}(c,2m-i), & \text{if\ } m \leq  i \leq 2m.
	\end{cases}
	\end{equation*}
	Hence, $\cD(F\circ 1_{\C},G\circ1_{\C})\leq n$.
\end{proof}

What follows is the categorical version of a well known result from Farber \cite{Farber}.

\begin{cor}\label{TCCAT2} 
$\cTC(\C)\leq \ccat(\C\times \C)$.
\end{cor}

\begin{proof}In Theorem \ref{CATDOM} take the functors $p_1,p_2\colon \C\times \C \to \C$. Then
$\cTC(\C)=\cD(p_1,p_2)\leq \ccat(\C\times \C)$.
\end{proof}


\subsection{Triangle Inequality}

\begin{prop}\label{prop:triangle_inequality_general} Let $F,G,H\colon \C \to \Dc$ be functors between $\C$ and $\Dc$ such that $\ccat(X)\leq 2$. Then 
	$$\cD(F,H)\leq \cD(F,G)+\cD(G,H).$$
\end{prop}

\begin{proof}
	First, notice that if two of the three functors are homotopic, then the result holds automatically, so assume that there is no pair of homotopic functors among $F,G$ and $H$. Since $\D(F,H) \leq \ccat(\C)$ (Corollary \ref{CATDOM}), the result follows.
\end{proof}

We do not know whether Proposition \ref{prop:triangle_inequality_general} holds without assumptions on the categorical category of the source category.

\subsection{Invariance}
We now prove the homotopy invariance of the homotopic distance.

\begin{cor}\label{prop:invariance_of_distance_under_homotopies}
\begin{enumerate}
	\item\label{COR1} Let $F,G\colon \C\to \Dc$ be functors and let $\alpha\colon \Dc \to \Dc^\prime$ be a functor with a left homotopy inverse. Then 
	$$\cD(\alpha\circ F,\alpha\circ G)=\cD(F,G).$$
	\item\label{COR2} Let $F,G\colon \C\to \Dc$ be functors and let $\beta\colon \C^\prime \to \C$ be a functor with a right homotopy inverse. Then 
	$$\cD(F\circ \beta,G\circ \beta)=\cD(F,G).$$
\end{enumerate}
\end{cor}

\begin{proof}
We prove \eqref{COR1} since \eqref{COR2} is analogous. By Proposition \ref{IZQ}, 
$$\cD(F,G)\geq \cD(\alpha \circ F,\alpha  \circ G)\geq \cD(\beta \circ \alpha \circ  F,\beta \circ \alpha \circ G).$$
But $\beta \circ \alpha \simeq 1_\Dc$ implies $\beta \circ \alpha \circ F \simeq F$ and $\beta \circ \alpha \circ G\simeq G$, hence
$\cD(\beta \circ \alpha  \circ F,\beta \circ \alpha \circ  G)=\cD(F,G)$ because the distance only depends on  the homotopy class.
\end{proof} 

\begin{prop}
	Assume $\alpha \colon \C \to \C'$ and  $\beta \colon \Dc \to \Dc'$ are homotopy equivalences between small categories,  connecting the functors $F\colon \C \to \Dc$ (resp. $G$) and $F^\prime\colon \C^\prime \to \Dc^\prime$ (resp. $G^\prime)$, that is, the following diagram is commutative:
	$$
\begin{tikzcd}
	\C \arrow[r,shift left, "F"] \arrow[r, shift right,"G"']   \arrow[d, "\alpha"] & \Dc \arrow[d, "\beta"] \\
	\C^\prime \arrow[r,shift left, "F^\prime"]\arrow[r,shift right,"G^\prime"'] &   \Dc^\prime
	\end{tikzcd}
$$
	Then $\cD(F,G)=\cD(F^\prime,G^\prime)$. 
\end{prop}

\begin{proof}
	We denote the homotopic inverse of $\beta$ by $\beta'$. Then from Corollary \ref{prop:invariance_of_distance_under_homotopies} it follows:
	$$\cD(F,G)=\cD(F\circ \alpha,G\circ \alpha)=\cD(\beta' \circ F\circ \alpha,\beta' \circ G\circ \alpha). \qedhere$$
\end{proof}

\begin{cor}
	$\ccat(X)$ and $\cTC(X)$ are homotopy invariant.
\end{cor}

Note that Corollary \ref{prop:invariance_of_distance_under_homotopies} generalizes the homotopy invariance of $\ccat$ stated by Tanaka in \cite{Tanaka3}.

\subsection{Products}
We study now the behaviour of the categorical homotopic distance under products.

\begin{theorem}\label{DISTPRODUCT}
 Given $F,G\colon \C\to \Dc$ and $F',G'\colon \C'\to \Dc'$, it is  
 $$\D(F\times F',G\times G')+1\leq \big(\D(F,G)+1\big) \cdot \big(\D(F',G')+1\big).$$
\end{theorem}

\begin{proof}
Given geometric coverings $\{\U_0,\ldots,\U_m\}$ and $\{\mathcal{V}_0,\ldots,\mathcal{V}_n\}$ of $\C$ and $\C'$, respectively,  such that $F_{\vert \U_i}\simeq G_{\vert \U_i}$ and $F'_{\vert \mathcal{V}_j}\simeq G'_{\vert \mathcal{V}_j}$, then it can be checked that $\{\U_i\times \mathcal{V}_j\}$ is a geometric cover of $\C\times \C'$ such that $F\times F'_{\vert \U_i\times \mathcal{V}_j}\simeq G\times G'_{\vert \U_i\times \mathcal{V}_j}$.
\end{proof}

\begin{example}\label{PRODLS}
    Set $F\colon \C\to \C$ and $F'\colon \C'\to \C'$ to be the identity functors and $G\colon \C\to \C$ and $G'\colon \C'\to \C'$ to be constant functors. Then $$\ccat(\C\times \C')+1 \leq (\ccat(\C)+1) \cdot (\ccat(\C')+1).$$
    Hence, Theorem \ref{DISTPRODUCT} generalizes the product inequality proved by Tanaka \cite{Tanaka3} for the categorical LS-category.
\end{example}

\begin{example}\label{PRODTC}
    Set $F\colon \C \times \C\to \C$ and $F'\colon \C' \times \C' \to \C'$ to be the projection functors onto the first factor and $G\colon \C \times \C\to \C$ and $G'\colon \C' \times \C' \to \C'$ to be the projection functors onto the second factor. Then $$\cTC(\C\times \C')\leq (\cTC(\C)+1) \cdot (\cTC(\C')+1)-1.$$
\end{example}

\subsection{Relationship between homotopic distances}

Ordinary homotopic distance between continuous maps and the two notions of categorical homotopic distance that we have defined so far are related by the following result:

\begin{prop}\label{prop:inequalities_homotopic_distances}
	Given two functors $F,G\colon \C\to \Dc$, then $$\D(\B F, \B G)\leq \wcD(F,G) \leq \cD(F,G).$$
\end{prop}

\begin{proof}
It is well known that given any subcomplex $Y$ of a CW-complex $X$, there exists an open neighborhood $U$ of $Y$ in $X$ such that $Y$ is a deformation retract of $U$ \cite{Hatcher}. Since deformation retracts are homotopy equivalences and the homotopic distance is invariant under homotopies \cite{QuiDa}, by Proposition \ref{prop:condition_being_geometric_cover} we have $\D(\B F, \B G)\leq \wcD(F,G)$. The fact that the classifying space functor preserves homotopies guarantees the inequality $\wcD(F,G) \leq \cD(F,G).$
\end{proof}

\begin{remark}
The difference between the categorical homotopic distance and the weak categorical homotopic distance can be arbitrarily large, as Example \ref{ex:poset_revirado_numeros_naturales} illustrates.
\end{remark}


\section{The context of posets}\label{sec:Posets}

As we explain below, a finite poset can be seen both as a small category and as a finite topological space. In tis way, order preserving maps between them can be seen both as functors and as continuous maps. Therefore, given two functors between posets $F,G\colon P \to Q$, it makes sense to study both their homotopic distance as continuous maps and their categorical homotopic distance as functors. We devote this section to the study of homotopic distance between order preserving maps.   For a more detailed exposition of the preliminaries on finite topological spaces we refer the reader to \cite{Barmak_book,Raptis,Stong}.


\subsection{Generalities on finite spaces and posets} From now on all posets are assumed to be finite. Recall that a poset $P$ can be seen as a small category where there is an arrow from the element $x$ to the element $y$ if and only if $x\leq y$.  Finite posets and finite topological spaces are in bijective correspondence.  If $(P, \leq)$ is a poset, a basis of a topology on $P$ is given by taking, for each $y\in P$ the set $$U_y:=\{x\in P\colon x\leq y\}.$$   
Conversely, if $X$ is a finite $T_0$-space, define, for each $x\in X$, the {\it minimal open set} $U_x$ as the intersection of all open sets containing $x$.  Then $X$ may be given a poset structure by defining $x\leq y$ if and only if $U_x\subset U_y$.  Moreover, functors (order preserving maps) between posets (seen as categories) are just the continuous maps between the associated topological spaces and the notion of homotopy between functors coincides with the topological notion of homotopy \cite{Raptis}. From now on, we will use the notions of poset and finite spaces interchangeably, and the same applies to functor, order preserving map and continuous map.

Given a poset $P$, we can consider the poset with the opposite order $P^{\op}$, which is the opposite category. Then,  the subposet $$F_x:=\{y\in P: y\geq x\}$$ of $P$ coincides with the subsposet $U_x$ of $P^{\op}$.

Given a poset $P$, its {\it order complex $\mathcal{K}(P)$} is the abstract simplicial complex whose simplices are  the non-empty chains of $P$. We say that the poset $P$ is a  {\it model} for the topological space $Y$ if the geometric realization $\vert \mathcal{K}(P) \vert$ of the simplicial complex  $\mathcal{K}(P)$ is homotopy equivalent to $Y$. 
Conversely, if $K$ is a simplicial complex, we associate a poset $\chi(K)$ to $K$ via the McCord functor $\chi$ \cite{McCord} where $\chi(K)=\{\sigma : \sigma \in K\}$ and $\sigma\leq \tau$ if and only if $\sigma$ is a face of $\tau$. 

We recall a classic result \cite{Barmak_book}:

\begin{lemma}\label{lema:contiguous_maps_in_posets_go_to_same_contiguity_class}
Given two homotopic continuous maps between posets $F,G\colon P\to P'$, then the simplicial maps $\kappa(F)$ and $\kappa(G)$ are in the same contiguity class. Conversely, let $\varphi,\phi \colon K\to L$ be simplicial maps which lie in the same contiguity class. Then $\chi(\varphi)\simeq \chi(\phi)\colon \chi(K)\to \chi(L)$.
\end{lemma}

\subsection{Homotopy equivalences in posets}

Stong \cite{Stong} showed that for any given finite poset $P$ the exists a unique subposet (up to isomorphism) $P'\subset P$, called the {\em core of $P$},  satisfying the following two conditions: 
\begin{itemize}
\item
first, $P^\prime$ is a deformation retract of $P$;
\item
second, no proper subposet of $P'$ is a deformation retract of $P$. 
\end{itemize} 
Under these circumstances $P'$ is called a   {\em minimal} poset.  

As a consequence of the homotopy invariance of the distance, it follows that in order to compute the (categorical) homotopic distance between functors $F,G\colon P\to Q$, where $P$ and $Q$ are posets, it is enough to study the (categorical) homotopic distance between the associated functors $F^\prime, G^\prime\colon P^\prime \to Q^\prime$ between the cores.

\begin{corollary}
Given two functors $F,G\co P\to Q$ between two finite posets $P,Q$, let   $P^\prime$ (respectively, $Q^\prime$) the core of $P$ (resp., of $Q$). Denote by $F',G^\prime\colon P^\prime \to Q^\prime$   the compositions of $F$ and $G$ with the equivalences $P\simeq P^\prime$ and $Q\simeq Q^\prime$.  Then:
$$\cD(F,G)=\cD(F',G')$$ and $$\D(F,G)=\D(F',G').$$
\end{corollary}

Therefore, from now on, we can restrict our attention to minimal spaces.

\subsection{Coverings and bounds} We begin  with a lemma which relates the notions of geometric cover and open cover of posets. 

\begin{lema}\label{lemma:open_is_geometric}
	If $P$ is a finite poset and $\{U_0,\dots,U_m\}$ is an open cover of $U$, then  $\{U_0,\dots,U_m\}$ is also a geometric cover.
\end{lema}

\begin{proof}
	Let $x_0\leq \cdots \leq x_n$ be a sequence of composable arrows in $P$. Since   there is a $U_k$ such that $x_n\in U_k$, by definition of the open sets in $P$ it is $U_{x_n}\subset U_k$, so  $x_0, \dots ,x_n\in U_k$.
\end{proof}

The following results help us to implement the computation of the (categorical) homotopic distance when the domain is a finite poset by reducing the open coverings we have to test.
 
\begin{proposition}\label{prop:form_of_coverings_D}
Given two finite posets $P$ and $Q$ and two functors $F,G\co P\to Q$, in order to compute $\D(F,G)$, which is finite, it is enough to study open coverings $\{U_i\}$ whose elements are of the form $U_i=U_{x_0}\cup \cdots \cup U_{x_{i_n}}$, where the $x_{k}$ are maximal elements (with respect to the order relation in $P$). 
\end{proposition}

\begin{proof}
	Given $x_k\in P$, the basic open subset $U_{x_k}$ is contractible (Proposition \ref{prop_terminal_initial_object_implies_zero_homotopic_distance}). Hence, $F_{\vert U_{x_k}}\simeq G_{\vert U_{x_k}}$. Therefore it is clear that an open cover $\U=\{U_i\}$ whose elements are of the form $U_i=U_{x_0}\cup \cdots \cup U_{x_{i_n}}$ where the $x_{k}$ are maximal elements  provides an open cover for studying $\D(F,G)$ and since the poset $P$ is finite, so is the cover and $\D(F,G)$. Now, we will prove that it is enough to study such coverings. Given an open cover  $\{V_i\}_{i=0}^n$ such that $F_{\vert V_i}\simeq G_{\vert V_i}$ we will obtain a cover formed by unions of maximal basic open sets with at most $n+1$ elements. Suppose that $V_i=\{x_1,\ldots,x_k\}$. Among the elements of $V_i$ pick the ones with are maximal elements of $P$, assume they are $\{x_{i_0},\ldots, x_{i_l}\}$, then define $U_i=U_{x_{i_0}}\cup \cdots \cup U_{x_{i_l}}$. Note that $U_i\subset V_i$. It can be checked that the covering constructed by this procedure $\{U_i\}$ satisfies that $F_{\vert U_i}\simeq G_{\vert U_i}$ and it is a cover of the required form.
\end{proof}

The proof of the following result is similar.

\begin{proposition}
Given two finite posets $P$ and $Q$ and two functors $F,G\co P\to Q$,  in order to compute $\cD(F,G)$, which is finite, it is enough to study geometric coverings $\{U_i\}$ whose elements are of the form $U_i=C_{i_0}\cup \cdots \cup C_{i_n}$, where the $C_{i_k}$ are maximal chains. 
\end{proposition}

As a consequence of the previous two results, we can given an upper bound for the categorical homotopic distance.

\begin{corollary}
	Given two finite posets $P$ and $Q$ and two functors $F,G\co P\to W$, then $\D(F,G)$ and $\cD(F,G)$ are less than or equal to the number of maximal elements of $P$. Furthermore, $\cD(F,G)$ is less than or equal to the number of minimal elements.
\end{corollary}

\subsection{Relations with other homotopic distances}

In \cite{SamQuiJa_2,SamQuiJa_TC,SamQuiJa_1}, simplicial versions of LS category and topological complexity were given by one of the authors, by replacing the notion of homotopic continuous maps with that of contiguous simplicial maps. Recall that two simplicial maps $\varphi,\psi\colon K \to L$ are said to be contiguous if for every simplex $\sigma \in K$, $\varphi(\sigma) \cup \psi(\sigma)$ is a simplex of $L$. Two simplicial maps $\varphi,\psi\colon K \to L$ lie in the same contiguity class if there exists a sequence $\varphi=\varphi_1,\ldots, \varphi_n=\psi$ such that $\varphi_i$ and $\varphi_{i+1}$ are contiguous for every $0\leq i < n$. 

In the same vein, a notion of distance between simplicial maps  can be defined.

\begin{definition}[\cite{QuiDa}]
	The {\em contiguity distance} $\sD(\varphi,\psi)$  between two simplicial maps $\varphi,\psi\colon K \to L$  is the least integer $n\geq 0$ such that there exists a covering of $K$ by subcomplexes $K_0,\dots,K_n$ such that the  restrictions $\varphi_{\vert K_j},\psi_{\vert K_j}\colon K_j \to L$ are in the same contiguity class, for all $j=0,\dots,n$.  If there is no such covering, we define $\sD(f,g)=\infty$.
\end{definition}

This notion of contiguity distance generalizes those of simplicial LS category $\scat(K)$ and discrete topological complexity $\sTC(K)$ \cite{SamQuiJa_2,SamQuiJa_TC,SamQuiJa_1,QuiDa,Scoville}: 

\begin{example}
	Given two simplicial complexes $K$ and $L$, denote by $K \prod L$ their categorical product \cite{Kozlov}. The contiguity distance between the projections $p_1,p_2\colon K \prod K\to K$ equals $\sTC(K)$, as follows from \cite[Theorem 3.4]{SamQuiJa_TC}.
\end{example}

\begin{example}
The simplicial LS category of a simplicial map between simplicial complexes  $\varphi\colon K \to L$, denoted $\scat(\varphi)$ \cite{Scoville}, is  the contiguity distance $\sD(\varphi,v_0)$  where $v_0\colon K \to L$ is a constant simplicial map.
\end{example}

 
%

\begin{theorem}\label{theorem:relations_between_distances_context_of_posets}
	Given order preserving maps between finite posets $F,G\colon P\to Q$, then $$\D(\B F, \B G)\leq \wcD(F,G) \leq  \sD(\kappa(F),\kappa(G)) \leq \cD(F,G) \leq \D(F,G).$$
\end{theorem}

\begin{proof}
From Proposition \ref{prop:inequalities_homotopic_distances} we already have the inequalities: $$\D(\B F, \B G)\leq \wcD(F,G) \leq \cD(F,G).$$ 
First, we prove that: $$\cD(F,G) \leq D(F,G).$$ 
Suppose that $D(F,G)=n$ with an open covering $\{U_0,\dots,U_n\}$. Because of Lemma \ref{lemma:open_is_geometric} the collection $ \{U_0,\dots,U_n\}$ is also a geometric covering. 

Now we show that:
 $$\sD(\kappa(F),\kappa(G)) \leq \cD(F,G).$$
Recall that a collection $\{\U_{\lambda}\}_{\lambda \in \Lambda}$ of subcategories of a category $\C$, is a geometric cover if and only if the collection of subcomplexes $\{\B \U_{\lambda}\}_{\lambda \in \Lambda}$ covers $\B\C$ (Proposition \ref{prop:condition_being_geometric_cover}). 
 Now, the inequality $\sD(\kappa(F),\kappa(G)) \leq \cD(F,G)$ follows from Lemma \ref{lema:contiguous_maps_in_posets_go_to_same_contiguity_class}. Finally, the inequality $\wcD(F,G) \leq  \sD(\kappa(F),\kappa(G))$ follows from the fact that if two simplicial maps are in the same contiguity class, then their geometric realizations are homotopic \cite{Barmak_book}.
\end{proof}


As a consequence of Theorem \ref{theorem:relations_between_distances_context_of_posets} we obtain several results relating existing notions for LS categories:

\begin{corollary}
Given a poset $P$, it is 
$$\cat(\vert\kappa(P)\vert)\leq \wccat(P) \leq  \scat(\kappa(P)) \leq \ccat(P) \leq \cat(P),$$ where $\vert\kappa(P)\vert$ denotes the geometric realization of the simplicial complex $\kappa(P)$.
\end{corollary}

Therefore, Theorem \ref{theorem:relations_between_distances_context_of_posets} generalizes the results of Tanaka \cite{Tanaka3} regarding the categorical LS-category. 

\subsection{Subdivisions}

We recall that given a poset  $P$, its barycentric subdivision can be defined  as $\sd(P)=(\chi\circ \kappa) (P)$ \cite{Barmak_book}. Moreover, this construction is functorial. We have seen that $\cD(F,G)\leq \D(F,G)$ (Theorem \ref{theorem:relations_between_distances_context_of_posets}). By subdividing the domain we can reverse this inequality.

\begin{lema}\label{lem:subdivi_changes_ineq}
		Given two order preserving maps $F,G\colon P\to Q$ between finite posets , it is 
		$$\D(\sd(F),\sd(G))\leq \cD(F,G).$$
\end{lema} 

\begin{proof}
	Assume that $\cD(F,G)=n$ with the geometric cover $\{\U_i\}$. Note that $\kappa(U_i)$ is a subcomplex of $\kappa(P)$ and therefore $\sd(U_i)=\chi\circ \kappa(U_i)$ is an open subset of $\sd(P)$. Moreover, $\{\sd(\U_i)\}$ is a cover of $\sd(P)$. Finally, since $F_{\vert \U_i}\simeq G_{\vert \U_i}$, from Lemma \ref{lema:contiguous_maps_in_posets_go_to_same_contiguity_class}  follows that $F_{\vert \sd(\U_i)}\simeq G_{\vert \sd(\U_i)}$. As a consequence, $\D(\sd(F),\sd(G))\leq \cD(F,G)$.
\end{proof}

Moreover, the inequality becomes an equality after enough subdivisions:

\begin{prop}\label{prop:ccat_and_cat_stabilise_and_coincide}
	Given two order preserving maps $F,G\colon P\to Q$ between finite posets,   there exists a natural number $k$ such that the $k$-iterated barycentric subdivision stabilizes the distances, that is,  
	$$\D(\sd^k(F),\sd^k(G))= \cD(\sd^k(F),\sd^k(G)).$$
\end{prop}
 
\begin{proof}
	From Theorem \ref{theorem:relations_between_distances_context_of_posets} and Lemma \ref{lem:subdivi_changes_ineq} it follows that: $$\cD(\sd(F),\sd(G))\leq \D(\sd(F),\sd(G))\leq \cD(F,G).$$ Therefore,  $$
	\lim\limits_{k\to \infty} \D(\sd^{k} (F),\sd^k (G))=\lim\limits_{k\to \infty} \cD(\sd^k(F),\sd^k(G)). \qedhere
	 $$
\end{proof}

Observe that both Lemma \ref{lem:subdivi_changes_ineq} and Proposition \ref{prop:ccat_and_cat_stabilise_and_coincide} generalize the corresponding results in the context of posets by Tanaka for the categorical LS-category \cite{Tanaka3}.

\begin{remark}
	Section \ref{sec:Posets} could be generalized both to the context of preordered sets and to acyclic categories and most results would hold. 
\end{remark}

\bibliographystyle{plain}
\bibliography{biblio5}

\end{document}